\documentclass[11pt,reqno]{amsart}
\usepackage{xifthen}
  
\usepackage{pkgfile}

\newcommand{\rank}{\text{rank}}

\newcommand{\BP}{{\sf BP}}
\newcommand{\MBP}{{\sf MBP}}
\DeclareSymbolFont{extraup}{U}{zavm}{m}{n}
\DeclareMathSymbol{\varheart}{\mathalpha}{extraup}{86}
\DeclareMathSymbol{\vardiamond}{\mathalpha}{extraup}{87}

\usepackage[colorinlistoftodos]{todonotes}

\begin{document}

\title{Rank-$k$ random graphs and finite type branching processes }

\author[S. \ Chakraborty]{Suman Chakraborty}
\address{Department of Mathematics and Computer Science,  Eindhoven University of Technology, Eindhoven, Netherlands}
\email{s.chakraborty1@tue.nl}

\author[K.\ Raaijmakers]{Kjell Raaijmakers}
\address{Department of Mathematics and Computer Science,  Eindhoven University of Technology, Eindhoven, Netherlands}
\email{k.raaijmakers@student.tue.nl}

\author[R.\ van der Hofstad]{Remco van der Hofstad}
\address{Department of Mathematics and Computer Science,  Eindhoven University of Technology, Eindhoven, Netherlands}
\email{rhofstad@win.tue.nl}

\date{\today}

\maketitle

\begin{abstract}
In this note, we investigate fundamental relations between exploration processes in random graphs, and branching processes. We formulate a class of models that we call {\em rank-$k$ random graphs}, and that are special in that their neighborhood explorations can be obtained by a {\em thinning} of multi-type branching processes. We show that any rank-2 random graph can be described in terms of thinning of a 2-type branching process, while for higher rank, it is not clear how many types are needed.
\end{abstract}

\section{Introduction}
Let $A=(a_{ij})_{i,j=1}^n$ be an $n \times n$ symmetric matrix with non-negative entries, and $\ell_n>0$ a parameter. consider the random graph $G_n(A)$ with vertex set $[n]:=\{1,2,\ldots,n\}$, in which the edge $(i,j)$ is present with probability $1-\exp{\left(-\ell_n a_{ij}\right)}$ independent of all other edges for $1\leq i\leq j\leq n$. (We allow self-loops in $G_n(A)$.) Let us now roughly state two well-known facts (the formal statements are given later): \textbf{Fact 1:} The distribution of the neighborhoods of any vertex in the random graph $G_n(A)$ can be coupled with a thinned \emph{$n$-type} branching process associated with the matrix $A$. \textbf{Fact 2:} If $\rank(A)=1$, then the distribution of the neighborhoods of any vertex in the random graph can be coupled to a thinned one-type branching process associated with $A$. \par

This paper started with the following question motivated by the two facts mentioned above: If $\rank(A)=2$, can we then couple the distribution of the neighborhoods of the vertices in the random graph $G_n(A)$ with some thinned multi-type branching process of at most $2$ types? In this article, we answer this question in the affirmative, and encouraged by this, we investigate two related questions: (a) if $\rank(A)=k$ for some $k\geq 3$, can we couple the neighborhood distribution of any vertex in $G_n(A)$ with some branching process where the number of types depends on $k$, but not on $n$? (b) what conditions on $A$ (other than the rank of $A$) will ensure that $G_n(A)$ can be coupled with some multi-type branching process having a small number of types (independent of $n$)?  \par

If $A$ is an adjacency matrix of some simple graph, and it has rank $k$, then we show that one can couple the distribution of the neighborhoods of the vertices in the random graph $G_n(A)$ with some thinned multi-type branching process of at most $O(2^{k/2})$ types. Let us point out that the relationship between the rank of adjacency matrices and other graph properties such as chromatic number has been extensively studied (see van Nuffelen \cite{van1976bound}, Alon-Seymour \cite{alon1989counterexample}, Nisan-Wigderson\cite{nisan1995rank} etc.). Our work sheds light on percolation of finite graphs whose adjacency matrix has rank $k$ (see Remark \ref{rem-bond-percolation} below), and the type space of its branching process approximation.

Next, to put our contributions in context, we briefly recap some of the relevant random graph models from the literature and provide formal statements of the facts that we mentioned earlier. 

\subsection{Inhomogeneous random graphs and branching processes} The systematic study of random graphs began with the seminal article by Erd\H{o}s and R\'enyi~\cite{erdHos1960evolution}, who studied uniformly chosen graphs $G(n,M)$ with $n$ labelled vertices and $M$ edges. Slightly before that, Gilbert~\cite{gilbert1959random} introduced a closely related model $G_n(p)$, in which each edge is present independently with probability $p$. Since then, these models of random graphs have been studied extensively, and have become a source of wonderful mathematical problems.

A feature of $G_n(p)$ and $G(n, M)$ is that they are homogeneous, that is, in their definitions, all the vertices play an equivalent role. On the other hand, most real-world networks contain some degree of inhomogeneity. Several studies were devoted to introducing and analyzing models of inhomogeneous random graphs. A natural generalization of the $G_n(p)$ model is $G_n(\{p_{ij}\}_{i,j=1}^n)$, which is a random graphs on the vertex set $[n]:=\{1,2,\ldots,n\}$, in which the edges are chosen independently. Here, for $1 \leq i\leq j \leq n$ the probability of $ij$ being an edge is equal to $p_{ij}$. As pointed out in Bollob{\'a}s, Janson, and Riordan~\cite{bollobas2007phase}, it seems difficult to obtain results for this class of random graphs without imposing some conditions on the edge probabilities $p_{ij}$'s. Having that said, a remarkable result on connectivity of such graphs was obtained by Alon~\cite{alon1995note} (this model is also mentioned in Bollob{\'a}s~\cite[Chapter 2]{bollobas1998random}). It is easy to see that the $G_n(A)$ model is equivalent to the model $G_n(\{p_{ij}\}_{i,j=1}^n)$ if we allow the $a_{ij}$ parameters to take the value $\infty$. 

Bollob{\'a}s, Janson, and Riordan~\cite{bollobas2007phase} introduced a highly general class of inhomogeneous random graphs. In these random graphs the average degree is $\Theta(1)$, and one of the important examples considered by them is $p_{ij} = \kappa(i/n,j/n)/n$, where $\kappa$ is an appropriate function on $[0,1]^2$. A common strategy in the analysis of a random graph where the average degree is $\Theta(1)$ is to relate the neighborhood distribution of the random graph to an appropriate branching process. Then the branching process is analysed to extract relevant information about the random graph. To see this connection, first let us recall that one can always associate a branching process to a non-negative symmetric matrix $A$.

\subsection*{Non-negative matrices and an associated branching process} For a given symmetric matrix $A=(a_{ij})\in \mathbb{R}_{+}^{n \times n}$ and $\ell_n \in \mathbb{R}_{+}$, it is possible to naturally associate an $n$-type branching process $\BP(i_0)$ as follows:
\begin{itemize}
    \item[$\rhd$] The root of the branching process is of type $i_0$.
    \item[$\rhd$] If a particle in the branching process has the type $i$, then it gives birth to $\text{Poi}\left(\ell_n a_{ij}\right)$ many offspring of type $j$, independently for $j\in [n]$.
\end{itemize}
If the root $i_0$ is chosen uniformly at random from $[n]$, then we denote the branching process by $\BP(i_0)$.

\paragraph{\bf Multi-type branching processes, marks and graphs.} Below, we will call types in $[n]$ {\em marks}, and think of them corresponding to vertices in the random graph $G_n(A)$. This allows us to distinguish the marks from the {\em types} in the multi-type branching process. To relate the branching process to the random graph, we always need to keep the mark information, leading us to study {\em marked multi-type branching processes.}

In terms of these marked branching processes, the neighborhoods in $G_n(A)$ can then be obtained by a {\em thinning} of this branching process, in the sense that vertices that have already appeared are thinned, together with  all their descendants. This relates the neighborhoods in $G_n(A)$ to an $n$-dependent branching process. While this is very interesting, the fact that the number of types in the multi-type branching process {\em depends on $n$} makes it difficult to use this relation. In this paper, we investigate when the multi-type branching process with $n$ types can be replaced by a marked multi-type branching process with a number of types that does not depend on $n$.

\paragraph{\bf Marked multi-type branching process notation.} Let us fix some notation. We consistently write $\BP(i_0)$ and $\BP$ for the branching process with $n$ types, where the marks and types are the same, and $i_0$ is the mark (or type) of the root. When relating these processes to multi-type processes with fewer types, we write $\BP_k(i_0)$ and $\BP_k$ to denote that there are $k$ types in our branching process. We further write $\MBP_k$ for the {\em marked} $k$-type branching process, where we think of the marks as corresponding to the vertices of the random graphs under investigation. We then use $\widetilde\MBP_k(i_0)$ to denote the {\em thinned} version of $\MBP_k(i_0)$, where vertices having repeated marks are removed with all their descendants. Thinning is the content of the next section.

\subsection{Thinned processes}\label{subsec:430pm27apr22} The thinned process $\widetilde{\BP}(i_0)$ is constructed from $\BP(i_0)$ in the following way. Starting from $i_0$, $\BP(i_0)$ is explored in a breadth-first manner, that is, all the members of a generation are explored before moving on to the next generation, and, whenever a particle is found whose type appeared before in the search, the particle is deleted along with all its descendants. The resulting process is then denoted by $\widetilde{\BP}(i_0)$.  

Now, let us consider the random graph $G_n(A)$. Then, for a vertex $i_0$, its neighborhood shells $\mathcal{N}_k(i_0)$ are defined as follows:
\begin{itemize}
    \item[$\rhd$] $\mathcal{N}_0(i_0) = i_0$.
    \item[$\rhd$] $\mathcal{N}_{k+1}(i_0)=\left\{j\in\left(\bigcup_{l=0}^k\mathcal{N}_l(i_0)\right)^c: j\leftrightarrow \mathcal{N}_k(i_0)\right\}$.
\end{itemize}
Here, for two subsets of vertices $A$ and $B$, $A\leftrightarrow B$ means that there is an edge between  $A$ and $B$. Note that $\BP(i_0)$ is a branching process with $n$-types. Then the following proposition is the formal statement of our \textbf{Fact 1}:
\begin{prop}[Graph exploration is thinned branching process]\label{prop:431pm15mar22}
Let $\widetilde{\mathcal{T}}_k$ be the set of the marks of the particles that appeared in the $k$-th generation of the multi-type branching process $\widetilde{\BP}(i_0)$ with $n$ types. Then, the sequence of sets $\{\widetilde{\mathcal{T}}_k\}_{k\geq 0}$ has the same distribution as $\{\mathcal{N}_k(i_0)\}_{k\geq 0}$.
\end{prop}

With Proposition \ref{prop:431pm15mar22} in hand, we can focus our attention on marked branching processes, where the thinning will allow us to return to the random graph using Proposition \ref{prop:431pm15mar22}. Thinning, in turn, requires marks, which thus naturally lead us to study marked branching processes. The aim is to reduce the graph exploration to a thinning of a marked $k$-type branching process with $k$ as small as possible. To understand this branching process, in turn, we may ignore the marks and focus only on the types. That reduces the complexity of the branching process considerably.

The proof of Proposition \ref{prop:431pm15mar22} is a modification of an argument in \cite{MR2213964}, and we postpone it to Section \ref{sec:1126am06may2022}. Interestingly, when $\rank{(A)}=1$, that is, when $A=\mathbf{a}\mathbf{a}'$ for some $\mathbf{a} \in \mathbb{R}_{+}^{n\times 1}$, the branching process $\BP(i_0)$ has the same distribution as a one-type marked branching process, denoted by $\BP_1(i_0)$, where 
\begin{itemize}
    \item[$\rhd$] the root of the branching process $\BP_1(i_0)$ receives the mark $i_0$.
    \item[$\rhd$] a particle with mark $i$ gives birth to $\text{Poi}\left(a_i\ell_n\sum_{i=1}^na_i\right)$ many offspring.
    \item[$\rhd$] each offspring of a particle receives mark $j$ with probability $\tfrac{a_j}{\sum_{i=1}^na_i}$. 
\end{itemize}
We can construct an $n$-type branching process from $\BP_1(i_0)$ by calling a particle with mark $i$ in $\BP_1(i_0)$ to be of type $i$. In line with our general notation, let us denote this by $\MBP_1(i_0)$, which we next show has the same distribution as $\BP(i_0)$: 


\begin{thm}[Marked 1-type BP is $n$-type BP]\label{prop:419pm27apr22}
$\MBP_1(i_0)$ has the same distribution as $\BP(i_0)$.
\end{thm}
\begin{rmk}[General structure]
\label{remark-general}
The equality in distribution in Theorem \ref{prop:419pm27apr22}, where a marked multi-type branching process with few types has the same law as an $n$-type branching process, will be the recurring theme of this paper. 
\hfill $\vardiamond$
\end{rmk}

As the careful reader might already have noted, Propositions \ref{prop:431pm15mar22} and \ref{prop:419pm27apr22} provide us with a prescription to obtain the neighborhood shells of $G_n(A)$ from an $1$-type branching process when $A$ has rank $1$, by thinning of vertices with the same marks. More precisely, one can consider the one-type branching process $\BP_1(i_0)$, then using Proposition \ref{prop:419pm27apr22}, by calling a mark $i$ particle to be of type $i$ we obtain a branching process that has the same distribution as $\BP(i_0)$. Finally, using the thinning procedure described in Section \ref{subsec:430pm27apr22}, the neighborhood shells of $G_n(A)$ can be constructed using Proposition \ref{prop:431pm15mar22}.

\subsection{Finite rank random graphs} We introduce a class of undirected inhomogeneous random graph models on the vertex set $[n]:=\{1,2,\ldots,n\}$. These graphs may have self loops but may \emph{not} have multiple edges. Let $\mathbb{R}_{+}$ denotes the set of all non-negative real numbers, and $\N$ the set of all positive integers.    
\begin{dfn}[Rank-$k$ random graphs]
\label{def:rank-k_random_graph}
   Let $n \in \mathbb{N}$. For a given symmetric matrix $A=(a_{ij})\in \mathbb{R}_{+}^{n \times n}$ and $\ell_n \in \mathbb{R}_{+}$, $G_n(A)$ is called a \textit{rank-$k$ random graph} on the vertex set $[n]$, if the edges in $G_n(A)$ are present independently, and the probability that the edge $(i,j)$ is present is given by 
    \begin{equation}
    \label{eqn:837amnov20}
        p_{ij} = 1-\exp{\left(-\ell_n a_{ij}\right)} \quad \forall  (i,j)\in [n]\times [n],
    \end{equation}
    where the matrix $A$ has rank $k$.
\end{dfn}
\begin{rmk}[Global sparsity parameter]
In Definition \ref{def:rank-k_random_graph}, $\ell_n$ can be interpreted as the global sparsity parameter for the model. Indeed, when $\ell_n$ is a positive number independent of $n$, then the expected number of edges in $G_n(A)$ is of the order $n^2$, and if $\ell_n \to 0$ as $n\to \infty$, then using Taylor's expansion it is easily seen that the expected number of edges is of order $\ell_nn^2$.\hfill $\vardiamond$
\end{rmk}
\begin{rmk}[Bond percolation models as a special case]
\label{rem-bond-percolation}
Note that if the entries of $A$ are either $0$ or $1$, then $A$ becomes an adjacency matrix of some graph $G$ on $n$ vertices, and \eqref{eqn:837amnov20} simply becomes the bond percolation model on the ground graph $G$, that is, every edge in $G$ is kept independently with  probability $1-\exp{\left(-\ell_n \right)}$.   \hfill $\vardiamond$
\end{rmk}
\begin{rmk}[Graphs and marked branching processes]
Recall Proposition \ref{prop:431pm15mar22}, which states that graph neighborhoods have the same distribution as thinned marked branching processes. This relation allows us to study rank-$k$ random graphs by only considering their marked multi-type branching process approximation. Thus, from now on, we will mainly work with branching processes, and only discuss random graphs when necessary.\hfill $\vardiamond$
\end{rmk}

\paragraph{\bf Main results of this paper.}
With Definition \ref{def:rank-k_random_graph} in hand, we can now explain the main results of this paper. We prove that $\BP(i_0)$ has the same distribution as a specific marked two-type branching process when $\rank(A)=2$. When $\rank(A)=k>2$, on the other hand, the situation is not so clear and we need to make further assumptions. When $A$ has random dot product structure, $\BP(i_0)$ has the same distribution as a marked $k$-type branching process. When, instead, $A$ has a  Kotlov-Lov\'asz structure, $\BP(i_0)$ has the same distribution as a marked $f_k$-type branching process where $f_k\leq O(2^{k/2})$.

\paragraph{\bf Organisation of this paper.}
In Section \ref{sec:positive_dot_prduct}, we start by treating the illustrative example of the random dot-product graphs, where we prove that it can be related to a $k$-type branching process. In Remark \ref{remark-collapsing}, we also show that we can think of this model as a rank-1 model on $kn$ vertices, where blocks of $k$ vertices are collapsed to 1. We use these resulte in Section \ref{section-rank2} to show that any rank-2 random graph can be related to a 2-type branching process through its spectral decomposition. In Section \ref{section-Kotlov-Lovasz}, we describe the Kotlov-Lov\'azs setting, which can be seen as  rank-1 percolation on a graph having a finite number of vertex types. In Section \ref{sec:1126am06may2022}, we prove Proposition \ref{prop:431pm15mar22}. In Section \ref{section-rankk-special}, we discuss a model related to that of the random dot-product random graph in Section \ref{sec:positive_dot_prduct}, where we allow also for negative contributions. We close in Section 
\ref{section-conclusion} with discussion and open problems.

\section{A special case: random dot-product setting}
\label{sec:positive_dot_prduct}
In this section we show that if $A$ is the sum of dot-products of vectors with non-negative entries, then we can couple $\BP(i_0)$ to a finite type branching process. This means that we assume that $A$ admits a decomposition of the form
\begin{equation}\label{eqn:511pm15mar22}
    A=\mathbf{v}_1\mathbf{v}_1'+ \mathbf{v}_2\mathbf{v}_2'+\ldots+\mathbf{v}_k\mathbf{v}_k',
\end{equation}
for some column vectors $\mathbf{v}_1, \mathbf{v}_2,\ldots,\mathbf{v}_k \in \mathbb{R}_{+}^{n}$.
We start with this example, since it nicely illustrates our main results, and also because the rank-2 setting in the next section can be reduced to it.

\begin{rmk}[Relation to random dot product graphs]
If $A$ is of the form \eqref{eqn:511pm15mar22}, then $G_n(A)$ is similar to the random dot-product model \cite{MR3827114} when $\ell_n\to 0$. Indeed, in the random dot-product model, $p_{ij}=\ell_n a_{ij}$
with $A$ as in \eqref{eqn:511pm15mar22}. 
By \cite{Jans08a}, our model is asymptotically equivalent to the random dot-product random graph when
    \begin{equation}
    \label{eqn:257pm15jul22}
        \sum_{1\leq i<j\leq n} (\ell_n a_{ij})^3=o(1),
    \end{equation}
which amounts to a mild sparsity condition.
\end{rmk}

Let $\mathbf{v}_l=(v_{l1},\ldots,v_{ln})'$ for $l\in [n]$, and note that $a_{ij}=\sum_{l=1}^k v_{li}v_{lj}$. Consider a $k$-type marked branching process, where the set of types is $\{t_1,t_2,\ldots,t_k\}$, and the set of marks is $[n]$. Then the marked $k$-type branching process is denoted by $\MBP_k(i_0)$, and is defined as follows: 
\begin{itemize}
    \item[$\rhd$] The root is marked with $i_0$.
    \item[$\rhd$] A particle with mark $i$ gives birth to $\text{Poi}(\ell_n v_{li}\sum_{j=1}^nv_{lj})$ many particles of type $t_l$ independently for $l=1,2,\ldots,k$.
    \item[$\rhd$] The $l$ type children of a particle independently receive a mark $j$ with probability $\frac{v_{lj}}{\sum_{j=1}^n v_{lj}}$.
\end{itemize}
We can obtain an $n$-type branching process from $\MBP_k(i_0)$ by calling a particle with mark $j$ to be of type $j$. Let us denote this $n$-type branching process by $\BP(i_0)$. The next proposition shows that it has the same distribution as $\MBP(i_0)$:  
\begin{thm}[Marked $k$-type BP is $n$-type BP]
\label{prop-n-vs-k-type-BP}
$\BP(i_0)$ has the same distribution as $\MBP_k(i_0)$.
\end{thm}

Theorem \ref{prop-n-vs-k-type-BP} is another example of the general principle in Remark \ref{remark-general}.

\begin{proof}
The root of both the branching processes are marked by $i_0$, and therefore they trivially have the same distribution. Let the distribution of $\BP(i_0)$ and $\MBP_k(i_0)$ be equal up to generation $l$. Now we show that, conditionally on generation $l$, the distributions of these processes are equal up to generation $l+1$. Consider a particle in generation $l$ in $\MBP_k(i_0)$, and suppose that it has mark $i$, which, in $\BP(i_0)$, means that the particle had mark $i$. Then the number of children of this particle that are of type $t_l$ and receive the mark $j$ has distribution
    \begin{equation}
    \text{Poi}(\ell_n v_{li}\sum_{j=1}^nv_{lj}\frac{v_{lj}}{\sum_{j=1}^n v_{lj}})=\text{Poi}(\ell_nv_{li}v_{lj})
    \end{equation}
independently for $l \in[k]$. Therefore, the distribution of the total number of children with mark $j$ has distribution             
    \begin{equation}
    \text{Poi}(\ell_n\sum_{l=1}^k v_{li}v_{lj})=\text{Poi}(\ell_na_{ij}).
    \end{equation}
Now, since the distribution of the marks are also independent, the particle gives birth to $\text{Poi}(\ell_na_{ij})$ many particle of mark $j$ independently for $j \in [n]$. Since in $\MBP_k(i_0)$ a particle with mark $j$ in $\BP(i_0)$ is called type $j$, the distributions of $\MBP_k(i_0)$ and $\BP(i_0)$ are equal. 
\end{proof}
Let us close this section by relating our results, in the random graph context, to collapsing of a rank-1 model on a larger vertex set:

\begin{rmk}[Rank-$k$ random graphs obtained through collapsing rank-$1$ models]
\label{remark-collapsing}
Note that the structure of $A$ in \eqref{sec:positive_dot_prduct} ensures that the rank of $A$ is at most $k$. It is possible to also obtain $G_n(A)$ using $k$ independent rank-$1$ random graphs. Indeed, let $\text{Explode}_i:=\{i_1,i_2,\ldots,i_k\}$ for $i\in [n]$. On the vertex set $V_j:=\{i_j\colon i\in [n]\}$, we construct the following random graph: for $h,l \in [n]$, there is an edge between $h_j$ and $l_j$ with probability  $1-\exp(-\ell_n v_{jh}v_{jl})$ independently of each other and everything else. We collapse the set $\text{Explode}_h$ to a single vertex $h$, so that there is an edge between $h$ and $l$ precisely when there was an edge between $\text{Explode}_h$ and $\text{Explode}_l$. From the construction, it is easy to see that the probability that there is an edge between $h$ and $l$ is equal to $1-\exp(-\ell_n \sum_{j=1}^k v_{jh}v_{jl})=1-\mathrm{e}^{-\ell_n a_{ij}}$.\hfill $\vardiamond$
\end{rmk}

\section{Rank-2 branching processes}
\label{section-rank2}
In this section we consider the case $\rank(A)=2$, and show that has the same distribution as a two-type branching process. Here, we make crucial use of the results for the random dot-product case (in {\bf Case II} below), to which we reduce our proof. Using the spectral decomposition of $A$ we can write 
\begin{equation}
    A= \lambda_1\mathbf{v}_1\mathbf{v}_1'+ \lambda_2\mathbf{v}_2\mathbf{v}_2',
\end{equation}
where $\lambda_1 \geq \lambda_2$ are the non-zero eigenvalues of $A$, and $\mathbf{v}_1$ and $\mathbf{v}_2$ are the corresponding eigenvectors. By the Perron-Frobenius theorem, all entries of $\mathbf{v}_1$ are non-negative, and $\lambda_1 \geq |\lambda_2|$. We distinguish two cases:
\medskip

\paragraph{\bf Case I of $\lambda_2<0$.} We start with a few observations. Firstly, we have that $a_{ij}=  \lambda_1v_{1i}v_{1j}+ \lambda_2v_{2i}v_{2j}$ for $i,j\in [n]$, and since $a_{ii}$ are non-negative for all $i\in [n]$, we get $\lambda_1v_{1i}^2+ \lambda_2v_{2i}^2 \geq 0$. Now using $\lambda_2<0$, we can write $\lambda_1v_{1i}^2\geq  |\lambda_2|v_{2i}^2 $. Therefore, $\sqrt{\lambda_1}v_{1i}\geq  \sqrt{|\lambda_2|}|v_{2i}|$, which in turn gives $\sqrt{\lambda_1}v_{1i}-  \sqrt{|\lambda_2|}v_{2i}\geq 0$ and $\sqrt{\lambda_1}v_{1i}+ \sqrt{|\lambda_2|}v_{2i}\geq 0$ for all $i\in [n]$. 

Denote $s_{+} := \tfrac{1}{\sqrt{2}}\sum_{i=1}^n (\sqrt{\lambda_1}v_{1i}+ \sqrt{|\lambda_2|}v_{2i})$, and   $s_{-} := \tfrac{1}{\sqrt{2}}\sum_{i=1}^n( \sqrt{\lambda_1}v_{1i}- \sqrt{|\lambda_2|}v_{2i})$.

We define a marked two-type branching process with marks $[n]$, and types $t_1$ and $t_2$. The process is denoted by $\MBP_2(i_0)$, where
\begin{itemize}
    \item[$\rhd$] the root is marked with $i_0$;
    \item[$\rhd$] a particle marked with $i$ gives birth to $\text{Poi}\left( \tfrac{\ell_n}{\sqrt{2}}(\sqrt{\lambda_1}v_{1i}- \sqrt{|\lambda_2|}v_{2i}) s_{+}\right)$ many children of type $t_1$, and $\text{Poi}\left( \tfrac{\ell_n}{\sqrt{2}}(\sqrt{\lambda_1}v_{1i}+ \sqrt{|\lambda_2}|v_{2i}) s_{-}\right)$ many children of type $t_2$;
    \item[$\rhd$] the $t_1$ type children of a particle independently receive a mark $j$ with probability  $\tfrac{1}{\sqrt{2}s_{+}}(\sqrt{\lambda_1}v_{1i}+ \sqrt{|\lambda_2|}v_{2i})$, and the $t_2$ type children of a particle independently receive a mark $j$ with probability $\tfrac{1}{\sqrt{2}s_{-}}(\sqrt{\lambda_1}v_{1i}- \sqrt{|\lambda_2|}v_{2i})$.
\end{itemize}
The next proposition shows that $\MBP_2(i_0)$ has the same distribution as $\BP(i_0)$:   
\begin{thm}[$n$-type branching process is marked two-type branching process]
$\MBP_2(i_0)$ has the same distribution as $\BP(i_0)$.
\end{thm}
\begin{proof}
The root of both the branching processes are marked by $i_0$, and therefore they trivially have the same distribution. Let the distribution of $\BP(i_0)$ and $\MBP_2(i_0)$ be equal up to generation $l$. Now we show that, conditionally on generation $l$, the distributions of these processes are equal up to generation $l+1$. Consider a particle in generation $l$ in $\MBP_2(i_0)$, and suppose that the particle had mark $i$. Then the number of children of this particle that are of type $t_1$, and receive the mark $j$ is 
\begin{equation}\label{eqn:248pm17mar22}
\text{Poi}\left( \tfrac{\ell_n}{\sqrt{2}}(\sqrt{\lambda_1}v_{1i}- \sqrt{|\lambda_2|}v_{2i}) s_{+} \tfrac{1}{\sqrt{2}s_{+}}(\sqrt{\lambda_1}v_{1j}+ \sqrt{|\lambda_2|}v_{2j})\right),
\end{equation}
and similarly the number of children of this particle that is of type $t_2$ and receives the mark $j$ is
\begin{equation}\label{eqn:249pm17mar22}
\text{Poi}\left( \tfrac{\ell_n}{\sqrt{2}}(\sqrt{\lambda_1}v_{1i}+ \sqrt{|\lambda_2|}v_{2i}) s_{-} \tfrac{1}{\sqrt{2}s_{-}}(\sqrt{\lambda_1}v_{1j}- \sqrt{|\lambda_2|}v_{2j})\right),
\end{equation}
Therefore, using \eqref{eqn:248pm17mar22} and \eqref{eqn:249pm17mar22}, the distribution of the total number of children with mark $j$ is 
 \begin{equation}
     \text{Poi}\left(\ell_n(\lambda_1v_{1i}v_{1j}- |\lambda_2|v_{2i}v_{2j})\right)= \text{Poi}\left(\ell_n(\lambda_1v_{1i}v_{1j}- |\lambda_2|v_{2i}v_{2j})\right) \stackrel{d}{=} \text{Poi}\left(\ell_na_{ij}\right).
     \end{equation}
Since the marks are also independent, the particle gives birth to $\text{Poi}(\ell_na_{ij})$ many particle of mark $j$ independently for $j \in [n]$. Since in $\MBP_2(i_0)$ a particle with mark $j$ has type $j$ in $\BP(i_0)$, the distribution of $\MBP_2(i_0)$ and $\BP(i_0)$ are therefore equal. 
\end{proof}
\medskip

\paragraph{\bf Case II: $\lambda_2>0$.}
In this case, $A$ is positive definite, and our aim is to show that here the problem boils down to a particular case of \eqref{eqn:511pm15mar22}. More precisely, we prove the following proposition that shows that this case can be reduced to the random dot-product case, the proof of which involves basic linear algebra and plane geometry:

\begin{thm}[Positive-definite matrices and random dot-products]
\label{prop:730pm17mar22}
Let $A$ be a symmetric positive-definite $n\times n$ matrix with non-negative entries and $\rank{(A)}=2$. Then there exists two non-negative vectors $\mathbf{w}_{1}$ and $\mathbf{w}_{2}$ such that  
\begin{equation}\label{eqn:728pm17mar22}
     A= \mathbf{w}_{1}\mathbf{w}'_{1} + \mathbf{w}_{2}\mathbf{w}'_{2}.
\end{equation}
\end{thm}

Note that we can write
$a_{ij} = \lambda_1v_{1i} v_{1j} + \lambda_2 v_{2i} v_{2j}$, for $i,j\in [n]$, which in turn can be written as
\begin{equation}\label{eqn:614pm17mar22}
     a_{ij} = \begin{pmatrix}
     \sqrt{\lambda_1}v_{1i}, &  \sqrt{\lambda_2} v_{2i}
    \end{pmatrix}
    \begin{pmatrix}
    \sqrt{\lambda_1}v_{1j}\\
    \sqrt{\lambda_2} v_{2j}
    \end{pmatrix} := x_i'x_j,
\end{equation}
where $x_i=(x_{1i}, x_{2i})'$. We show that the model can be transformed into a model for which $A$ admits a random dot-product decomposition as in \eqref{eqn:511pm15mar22} with $k=2$. 

First note that in \eqref{eqn:614pm17mar22} if we replace $x_i$ by $Wx_i$, where $W$ is an orthogonal matrix then the model remains the same. Moreover, since $a_{ij}\geq 0$, we we have $x'_ix_j\geq 0$ for all $i,j \in [n]$. 
\begin{lem}[Rotating vectors to non-negative cone]
\label{lem:708pm17mar22}
	Let $x_1,x_2,\ldots,x_n \in \mathbb{R}^2$ be points on the plane such that $x'_ix_j\geq 0$ for all $i,j \in [n]$. Then there is an orthogonal transformation such that the entries of $Wx_i$'s are non-negative for all $i\in[n]$.
\end{lem}
\begin{proof}
	Note that $x'_ix_j\geq 0$ implies that $0\leq \theta_{ij} \leq \frac{\pi}{2}$, where $\theta_{ij}$ is the angle between the vectors $x_i$ and $x_j$. Now starting from $x_1$, we search for all the next vector that are anticlockwise located within an angle $\pi/2$, and call the vector with largest (anticlockwise) angle with $x_1$ to be $x_L$. Similarly, we consider all the vectors within angle $\pi/2$, and call $x_R$ the vector with largest (clockwise) angle with $x_1$. Note that all vectors lie in between these two vectors. Also the angle between $x_L$ and $x_R$ (anticlockwise from $x_R$) is less than or equal to $\pi/2$. Therefore ,we can rotate all the points using a transformation $W$, say, that puts all vectors $x_1,x_2,\ldots,x_n$ in the positive orthant, as required.
\end{proof}

With Lemma \ref{lem:708pm17mar22} in hand, we are ready to prove Theorem \ref{prop:730pm17mar22}:

\begin{proof}[Proof of Theorem \ref{prop:730pm17mar22}]
Using Lemma \ref{lem:708pm17mar22}, and \eqref{eqn:614pm17mar22}, we get that $a_{ij}= (Wx_i)'Wx_j$ for all $i,j \in [n]$, and that the entries of $Wx_i$ are non-negative. Let $Wx_i = (w_{1i}, w_{2i})'$, and $\mathbf{w}_{i} = (w_{i1}, w_{i2},\ldots, w_{i2})'$, for $i=1,2$. It is now easy to check that $A= \mathbf{w}_{1}\mathbf{w}'_{1} + \mathbf{w}_{2}\mathbf{w}'_{2}$.
\end{proof}

\section{Rank $k$: Kotlov-Lov{\'a}sz model}
\label{section-Kotlov-Lovasz}
In this section, we discuss a model that can be considered a rank-1 percolation on a finite graph. We start by describing the set up.

\subsection{Adjacency matrix set up}
In this section we start by considering the case when the entries of $A$ are either $0$ or $1$, and all the diagonals are $0$. The results in this section crucially rely on a beautiful theorem by Kotlov and Lov{\'a}sz \cite{kotlov1996rank}, and, inspired by this, we call $G_n(A)$ to be $k$-Kotlov-Lov{\'a}sz model when $A$ has $0-1$ entries with diagonals equal to zero and $\rank(A)=k$. \par

  We first recall a result from Kotlov and Lov{\'a}sz \cite{kotlov1996rank} that will be essential in our analysis. A pair of non-adjacent vertices in a graph with the same neighbors is called a pair of {\em twin points}. The following is the main result in \cite{kotlov1996rank}:
\begin{thm}[Rank adjacency matrix of twin-free graph]\label{thm:1129pm18mar22}
Let $G$ be a twin-free simple and undirected graph, and let $k$ the rank of its adjacency matrix. Then
\begin{equation}
    \#\text{vertices of }G \leq O(2^{k/2}):=f_k.
\end{equation}
\end{thm}
Theorem \ref{thm:1129pm18mar22} has the following corollary that will be useful for our purposes:
\begin{cor}[Partitioning of graph with rank-$k$ adjacency matrix]\label{cor:1223pm18mar22}
Let $G$ be a simple and undirected graph on $n$ vertices, and let $k$ be the rank of its adjacency matrix. Then the vertices of $G$ can be partitioned into at most $f_k=O(2^{k/2})$ parts such that all the vertices in a class have identical neighbors. 
\end{cor}

\begin{proof}
The relation that a pair of vertices are twin is an equivalence relation. Therefore we can always partition the vertex set into say, $M_n$ parts where each class is an independent set, and all the vertices in a class have identical neighbors. 
We can form a subgraph by taking one vertex from each of these $M_n$ classes, and the adjacency matrix of that subgraph must have rank less than or equal to $k$. Since the subgraph is twin-free as well, using Theorem \ref{thm:1129pm18mar22} we get that $M_n \leq f_k$.
\end{proof}
\begin{rmk}[Edge structure Kotlov-Lov{\'a}sz partition]
\label{rmk:503pm26apr22}
Note that if we consider the vertex partition as in Corollary \ref{cor:1223pm18mar22}, then between each pairs of vertex classes  either {\em all} edges are present or {\em no} edge is present. We call this a \emph{canonical partition}, and label the classes by $P_1, P_2,\ldots, P_{f_k}$. Moreover, we construct an $f_k\times f_k$ proxy adjacency matrix $\tilde{A}= (\tilde{a}_{ij})_{i,j=1}^{f_k}$ such that $\tilde{a}_{ij}= 1$ if all edges are present between vertex class $i$ and $j$, and set $\tilde{a}_{ij}= 0$ otherwise.\hfill$\vardiamond$
\end{rmk}

If we have a symmetric matrix $A$ with entries in \{0,1\} such that $\rank(A)\leq k$, then using the observations in Remark \ref{rmk:503pm26apr22}, we can partition the rows of the matrix in at most $f_k$ classes. Let us label the rows of $A$ by $[n]$, and denote the above-mentioned partition by $P_1, P_2,\ldots, P_{f_k}$. Also let the number of rows in the class $P_i$ be $n_{i}$ for $i\in[f_k]$. Define the function $\mathcal{P}\colon [n]\to [f_k]$, such that $\mathcal{P}(i)=j$ if row $i$ belongs to class $P_j$. We now define an $f_k$-type marked branching process $\MBP_{f_k}(i_0)$ as follows:

\begin{itemize}
    \item[$\rhd$] The root has mark $i_0$;
    \item[$\rhd$] If a particle has mark $i$, then it gives birth to $\text{Poi}(n_{j}\ell_n \tilde{a}_{\mathcal{P}(i)j})$ many children of type $j$ for $j\in[f_k]$;
    \item[$\rhd$] All the type $j$ children of a particle are marked with a number from $P_j$ independently and uniformly with probability $1/n_{j}$. 
\end{itemize}


\begin{thm}[Marked $f_k$-type BP is $n$-type BP]\label{prop:2111pm27apr22}
$\MBP_{f_k}(i_0)$ has the same distribution as $\BP(i_0)$.
\end{thm}


\subsection{A rank-1 perturbation} Using a rank-$1$ perturbation, we can generalize the $k$-Kotlov-Lov{\'a}sz model to the case where the entries of $A$ are not necessarily in $\{0,1\}$, as follows: Suppose that $A$ is a $\{0,1\}$ matrix of rank $k$, and all the diagonal entries are equal to zero. Then consider the perturbed version of $A$, that is, the new matrix 
   \[
    B:= \diag(a_1,\ldots,a_n)A\diag(a_1,\ldots,a_n),
    \]
where $\diag(a_1,\ldots,a_n)$ is a diagonal matrix with entries $a_1,a_2,\ldots,a_n \geq 0$. Denote $\mathbf{a}= (a_1,\ldots,a_n)'$. Let us now define a marked $f_k$-type branching process $\MBP_{f_k}(i_0)$ as follows:

\begin{itemize}
    \item The root has mark $i_0$;
    \item If a particle has mark $i$, then it gives birth to $\text{Poi}\left(\ell_n a_i \tilde{a}_{\mathcal{P}(i)j}\sum_{i \in P_j}a_i\right)$ many children of type $j$ for $j\in[f_k]$;
    \item All type $j$ children of a particle receive a mark, say, $m$, from $P_j$ independently with probability $a_m/\sum_{i \in P_j}a_i$. 
\end{itemize}

\begin{thm}[Marked $f_k$-type BP is $n$-type BP]\label{prop:219pm27apr22}
$\MBP_{f_k}(i_0)$ has the same distribution as $\BP(i_0)$.
\end{thm}
\begin{proof}[Proof of Theorem \ref{prop:2111pm27apr22} \& \ref{prop:219pm27apr22}]
The proof follows  from the construction of $\MBP_{f_k}(i_0)$. We omit the details.
\end{proof}

\section{Rank-$k$: a special case}
\label{section-rankk-special}
In this section we consider a case that is similar to the case considered in Section \ref{sec:positive_dot_prduct} but our assumptions are different. Firstly, we assume that $A$ admits the decomposition
\begin{equation}\label{eqn:511pm05jun22}
    A=\mathbf{v}_1\mathbf{v}_1'+ \mathbf{v}_2\mathbf{v}_2'+\ldots+\mathbf{v}_k\mathbf{v}_k'.
\end{equation}
This is identical to the random dot-product case in \eqref{eqn:511pm15mar22}, aside from the fact that we no longer assume that the vectors $\mathbf{v}_i$ are non-negative.
Then we have $a_{ij}=\sum_{l=1}^k  v_{li}v_{lj}$, and we assume that $a_{ij}\geq 0$ for all $i,j\in[n]$. 
Contrary to the random dot-product setting in Section \ref{sec:positive_dot_prduct}, we show that we can relate the $n$-type branching process to a multi-type branching process with $2^{k-1}$ types. 
We make the following assumption:
\medskip
 
\paragraph{\bf Non-negative decomposition} We assume that \eqref{eqn:511pm05jun22} satisfies
\begin{equation}
    \label{eqn:541pm05jun22}
    v_{1i} \geq \sum_{l=1}^k |v_{li}| \qquad \text{for all} \qquad i\in[n].
\end{equation}
With a slight abuse of notations, let us define the types
\begin{equation}
    t_i(l):= v_{1i} \pm v_{2i} \pm \cdots \pm v_{li},
\end{equation}
for $i\in[n]$, and where $l$ runs over $2^{k-1}$ possible distinct configurations of $+$'s and $-$'s. The condition \eqref{eqn:541pm05jun22} ensures that $t_i(l)\geq 0$ for all $i$ and $l$. Now we are well equipped to define a $2^{k-1}$-type marked branching process $\MBP_{2^{k-1}}(i_0)$ as follows: 
\begin{itemize}
    \item The root is marked with $i_0$;
    \item A particle with mark $i$ gives birth to $\text{Poi}(\ell_n t_i(l)\sum_{j=1}^k t_i(l))$ many particles of type $t_l$ independently for $l\in[2^{k-1}]$;
    \item The $l$ type children of a particle independently receive a mark $j \in[n]$ with probability $\frac{t_j(l)}{\sum_{j=1}^n t_j(l)}$.
\end{itemize}

We again see that $\MBP_{2^{k-1}}(i_0)$ has the same law as $\BP(i_0)$:
 
\begin{thm}[Marked $2^{k-1}$-type BP is $n$-type BP]
\label{prop:804pm05jun22}
$\MBP_{2^{k-1}}(i_0)$ has the same distribution as $\BP(i_0)$.
\end{thm}
\begin{proof}
Consider a particle in $\MBP_{2^{k-1}}(i_0)$, and suppose that it is of type $i$, that is, in $\BP(i_0)$ the particle had mark $i$. Then the number of children of this particle that are of type $t_l$, and receive the mark $j$, is $\text{Poi}\left(\tfrac{1}{2^{k-1}}\ell_n t_i(l)\sum_{i=1}^{2^{k-1}} t_i(l)\frac{t_j(l)}{\sum_{j=1}^n t_j(l)}\right)$ disatributed, that is a $\text{Poi}(\ell_n t_i(l) t_j(l))$ distribution independently for $l \in[k]$. Therefore, the distribution of the total number of children with mark $j$ is a $\text{Poi}\left(\tfrac{1}{2^{k-1}}\ell_n\sum_{l=1}^{2^{k-1}} t_i(l) t_j(l)\right)$. A simple calculation, that we state in Lemma \ref{lem:823pm05jun22}, shows that this is equal to a $\text{Poi}(\ell_na_{ij})$ distribution (since $\sum_{l=1}^{2^{k-1}} t_i(l) t_j(l) = \sum_{l=1}^k  v_{li}v_{lj} $). Now, since the distribution of the marks are also independent, the particle gives birth to $\text{Poi}(\ell_na_{ij})$ many particle of mark $j$ independently for $j \in [n]$. The rest of the proof follows from an induction argument similar to the proof of Theorem \ref{prop-n-vs-k-type-BP}.
\end{proof}
Let us now state the technical lemma that we have used in the proof of Theorem \ref{prop:804pm05jun22}. It can be proved using induction and we omit further details. 
\begin{lem}[Alternating sums]
\label{lem:823pm05jun22}
For any two vectors $\mathbf{a}, \mathbf{b} \in \mathbb{R}^k$, $k\geq 1$ let
\begin{align*}
    S_k(\mathbf{a}, \mathbf{b}) &:= (a_1+a_2+\cdots+a_k)(b_1+b_2+\cdots+b_k) + (a_1-a_2+\cdots+a_k)(b_1-b_2+\cdots+b_k) \\
    &+ \cdots + (a_1-a_2-\cdots-a_k)(b_1-b_2-\cdots-b_k), \nonumber
\end{align*}
where the sum is over all $\pm$ coefficients of $(a_2,b_2), \ldots, (a_k,b_k) $. Then
\begin{equation*}
    S_k(\mathbf{a}, \mathbf{b}) = 2^{k-1} (a_1 b_1 + \cdots + a_k b_k).
\end{equation*}
\end{lem}


\section{Random graphs as thinned branching processes: Proof of Proposition \ref{prop:431pm15mar22}}
\label{sec:1126am06may2022}

Note that $\widetilde{\BP}(i_0)$ has type-$i$ root, and $\mathcal{N}_0(i_0)=i_0$. Therefore, the proposition is true for $k=0$. Now suppose that the proposition is true up to generation $k$. Pick a particle from generation $k$ in $\widetilde{\BP}(i_0)$, and suppose the particle is of type $i$. Then it has at least one child of type $j$ with probability $1-\exp{\left(-\ell_n a_{ij}\right)}$ independently for all $j\in [n]$. If $j$-type children already appeared in $\widetilde{\BP}(i_0)$ within generation $k$ then we delete all the $j$ type children from generation $k+1$. If $j$-type children have not appeared in $\widetilde{\BP}(i_0)$ within generation $k$, and we get multiple $j$-type children then we just keep one of them. On the other hand, in the graph we consider the vertex $i$ in $\mathcal{N}_k(i_0)$. Then it has a neighbor $j$ in $\mathcal{N}_{k+1}(i_0)$ with probability $1-\exp{\left(-\ell_n a_{ij}\right)}$ independently for each $j$ only if the vertex $j$ has not appeared in $\cup_{\ell=0}^k\mathcal{N}_{\ell}(i_0)$. This proves that the proposition is true up to generation $k+1$. Therefore the proposition is proved using induction.
\qed

\section{Conclusion, discussion and open problems}
\label{section-conclusion}
In this section, we discuss our results and state some further open problems.
\medskip

\paragraph{\bf Linking two fundamental notions}
The rank of a matrix is a fundamental property of that matrix. In turn, the branching process limit of a random graph is a fundamental property of the random graph. Our paper nicely links these two fundamental notions, in a quantitative manner. One can expect that such a relation has deep consequences, and below we discuss a few of them. 
\medskip 

\paragraph{\bf Local limits and giants of dot-product random graphs}
Bollob\'as, Janson and Riordan \cite{bollobas2007phase} prove general convergence results of local neighborhoods in inhomogeneous random graphs to branching processes, and, as a result of this, identify when there is a giant (a component of size order of $n$), and if so, how large it is. Consider the model in \eqref{eqn:511pm15mar22} for $k=2$. Then the largest eigenvalue of $A$ is 
\begin{equation}
\label{eqn:312pm15jul12}
    \lambda_1(A):=\tfrac{1}{2}(\sum_{i=1}^n v_{1i}^2+ \sum_{i=1}^n v_{2i}^2)+\tfrac{1}{2}\sqrt{(\sum_{i=1}^n v_{1i}^2 - \sum_{i=1}^n v_{2i}^2)^2 +2(\sum_{i=1}^n v_{1i}v_{2i})^2 }.
    \end{equation}
When \eqref{eqn:257pm15jul22} is satisfied, the results in \cite{bollobas2007phase, Jans08a}, give that if $\ell_n>1/\lambda_1(A)$, there is a giant in $G_n(A)$ with high probability, assuming that the empirical distributions of $(v_{1i}, v_{2i})_{i\in[n]}$ are sufficiently regular. Our result describes the local neighborhoods of $G_n(A)$ in terms of a $2$-type marked branching process (as opposed to a branching process with arbitrarily many types in \cite{bollobas2007phase}).    
\medskip
\paragraph{\bf Coupling beyond the sparse regime}
The above shows that our results have important consequences in the sparse setting. However, the beauty of the result is that the connection to finite-type branching processes holds {\em irrespective} of the edge density. One can imagine that stochastic domination by an explicit finite-type branching process is useful far beyond the sparse setting. For example, for the {\em connectivity} of random graphs, the typical threshold occurs in general inhomogeneous random graphs when the average degree is of order $\log{n}$, see \cite{DevFra14} and the references therein.
\medskip
\paragraph{\bf Critical behavior of rank-$k$ random graphs}
The general structure in Bollob\'as, Janson and Riordan \cite{bollobas2007phase} allows one to identify exactly when there is a giant or not, as well as several other properties of the random graphs involved. However, this set-up is likely not suitable to decide what the precise critical behavior is. Already in the rank-1 case, there are several universality classes (see \cite[Chapter 4]{Hofs23} and the reference therein for an extensive overview), and one can expect that this only becomes more tricky when the rank is higher. For example, in the dot-product setting, is the heaviest weight deciding the universality class, or is there some collaboration between the various weights? Recall \eqref{eqn:312pm15jul12}, and note that when $v_{1i}$'s and $v_{1i}$' are i.i.d. samples from random variables $V_1$ and $V_2$ with finite second moment, then $\lambda_1(A)/n$ converges to $\tfrac{1}{2}(\mathbb{E}(V_1^2)+ \mathbb{E}(V_2^2))+\tfrac{1}{2}\sqrt{(\mathbb{E}(V_1^2) - \mathbb{E}(V_2^2))^2 +2(\mathbb{E}(V_1 V_2))^2 }$. In the light of \cite[Chapter 4]{Hofs23}, it appears that the moments of $V_1$ and $V_2$ will play a role to describe the critical behavior of $G_n(A)$.
See \cite{Mier08} for a nice description how critical multi-type branching processes can be explored in terms of 1-type branching processes. Such an approach has the potential to be extended to the random graph setting when the description is in terms of multi-type branching processes with finitely many types independent of $n$. This may give a starting point for the analysis of the critical behavior of rank-$k$ random graphs.
\medskip
\paragraph{\bf Ranks of random graphs and branching processes: optimality}
Obviously, our results are not sharp for rank-$k$ random graphs when $k\geq 3$. It is unclear to us whether the rank really determines how many types the multi-type branching process description needs, and we see several different cases: For the random dot-product setting, $k$ suffices, for the Kotlov-Lov\'asz setting, $f_k=O(2^{k/2})$ suffices, while in the setting in \eqref{eqn:511pm05jun22}, we need $2^{k-1}$ types. It will be interesting to see in more detail what the minimal number is, and how this depends on the precise structure of $A$.

\medskip
\paragraph{\bf Acknowledgement.} The work of SC and RvdH is supported in part by the Netherlands Organisation for Scientific Research (NWO) through the Gravitation {\sc Networks} grant 024.002.003.

\bibliographystyle{plain}
\bibliography{referencesrankk.bib}
\end{document}